\documentclass{amsart}
\usepackage[all]{xy}\usepackage{amssymb,latexsym}
\newtheorem{theorem}{Theorem}
\newcommand{\bt}{\begin{theorem}}
\newcommand{\et}{\end{theorem}}
\newtheorem{lemma}{Lemma}
\newcommand{\bl}{\begin{lemma}}
\newcommand{\el}{\end{lemma}}
\newcommand{\bmat}{\left(\begin{matrix}}
\newcommand{\emat}{\end{matrix}\right)}
\newcommand{\bsmallmat}{\left(\begin{smallmatrix}}
\newcommand{\esmallmat}{\end{smallmatrix}\right)}

\title{The Hermite-Sylvester criterion for real-rooted polynomials}
\author{Melvyn B. Nathanson}
\address{Department of Mathematics\\Lehman College (CUNY)\\Bronx, NY 10468}
\email{melvyn.nathanson@lehman.cuny.edu}
\thanks{Supported in part by a grant from the PSC-CUNY Research Award Program.}
\date{\today}

\begin{document}

\begin{abstract}
A polynomial is \emph{real-rooted} if all of its roots are real.
This note gives a simple proof of the Hermite-Sylvester theorem that 
a  polynomial $f(x) \in {\mathbf R}[x]$  is real-rooted 
if and only if an associated quadratic form is positive semidefinite. 
\end{abstract}

\maketitle

A polynomial is \emph{real-rooted} if all of its roots are real. 
%(For applications of real-rootedness, see P. Br\" and\' en~\cite{bran15}.)
Hermite and Sylvester~\cite{vond16}  proved that  a  polynomial $f(x) \in {\mathbf R}[x]$  is real-rooted 
if and only if a certain quadratic form is positive semidefinite. 
Here is a short proof of the Hermite-Sylvester theorem that uses only Lagrange interpolation. 

The \emph{quadratic form}\index{quadratic form}  associated with a real 
$n \times n$ matrix $H = \bmat h_{i,j} \emat$ 
is 
\[
Q = Q(x_1,\ldots, x_n) = \sum_{i=1}^n \sum_{j=1}^n h_{i,j} x_i x_j.
\]
This form is \emph{positive semidefinite}\index{positive semidefinite form} 
if $Q(x_1,\ldots, x_n) \geq 0$ for all vectors $\bsmallmat x_1 \\ \vdots \\ x_n \esmallmat \in {\mathbf R}^n$.

Let $f(x)$ be a polynomial of degree $n$, and let $\lambda_1,\ldots, \lambda_n$ be the 
(not necessarily distinct) roots of $f(x)$.  
Define the \emph{$k$th power sum}\index{ power sum}  
\[
m_k = m_k(\lambda_1,\ldots, \lambda_n) = \sum_{\ell=1}^n \lambda_{\ell}^k.
\]
Newton's identities~\cite{zeil84} enable the efficient computation of the numbers $m_k$ 
from the coefficients of $f(x)$.
The \emph{Hermite matrix}\index{Hermite matrix} constructed from 
the polynomial $f(x)$ is the $n\times n$ matrix 
\begin{align*}
H_f   & = 
\bmat 
m_0 & m_1 & m_2 & \cdots & m_{n-1} \\
m_1 & m_2 & m_3 & \cdots & m_n \\
m_2 & m_3 & m_4 & \cdots & m_{n+1} \\
\vdots &&&& \vdots \\
m_{n-1} & m_n & m_{n+1} & \cdots & m_{2n-2} 
\emat  = \bmat h_{i,j} \emat 
\end{align*} 
where 
$
h_{i,j} = m_{i+j-2}
$
for all $i,j \in \{ 1,\ldots, n \}$.

Let $ \overline{\lambda}$ denote the complex conjugate of the complex number $\lambda$.

\begin{lemma}           \label{perm:lemma:Lagrange-0}
Let $\Lambda$ be a  nonempty finite set of $r$  complex numbers that 
is closed under complex conjugation, that is, 
\[
\Lambda = \left\{ \overline{\lambda}: \lambda  \in \Lambda \right\}.
\]
Let %$|\Lambda| = r$ and let $
\[
p(t) = \sum_{j=0}^{r-1} c_j   t^j \in  {\mathbf C}[t]
\]
 be a polynomial of degree at most $r-1$.  If
\begin{equation}                    \label{perm:Lagrange-0}
p(\overline{\lambda} ) = \overline{p(\lambda)} 
\end{equation}
for all $\lambda \in \Lambda$, then $p(t) \in {\mathbf R}[t]$.
\end{lemma}

\begin{proof}
We use the fact that if polynomials $p(t) \in  {\mathbf C}[t]$ and $q(t) \in  {\mathbf C}[t]$ 
have degrees at most $r-1$,  
and if $p(\lambda) = q(\lambda)$ for all $\lambda$ in a set of size $r$, 
then $p(t) = q(t)$.
Let   
\[
% p(t) = \sum_{j=0}^{r-1} c_j   t^j \qquad\text{and}\qquad  
q(t) = \sum_{j=0}^{r-1} \overline{c_j} \  t^j.
\]
For all $\lambda \in  {\mathbf C}$, we have $\overline{ \overline{ \lambda}} = \lambda$.  
For all $\lambda \in \Lambda$, we have $\overline{\lambda} \in \Lambda$, and so 
\begin{align*}
% \sum_{j=0}^{r-1} c_j \lambda^j & = 
 p(\lambda) & = p \left(\overline{ \overline{ \lambda}} \right) =  \overline{  p \left(  \overline{ \lambda}\right)  } 
 = \overline{ \sum_{j=0}^{r-1} c_j \overline{ \lambda}^j } 
   =  \sum_{j=0}^{r-1}  \overline{c_j}  \overline{ \overline{ \lambda}}^j  =  \sum_{j=0}^{r-1}  \overline{c_j}  \lambda^j 
 =  q(\lambda).
\end{align*}
Therefore, $p(t) = q(t)$, and so, for all $j \in \{0,1,\ldots, r-1\}$, 
we have  $c_j = \overline{c_j} \in {\mathbf R}$ and $p(t) \in {\mathbf R}[t]$.  
This completes the proof.   
\end{proof}

\begin{lemma}                   \label{perm:lemma:Lagrange-1}
Let $\Lambda  = \{\lambda_1,\lambda_2,\lambda_3,\ldots, \lambda_r\}$ 
be a nonempty finite set of $r$ complex numbers that 
is closed under complex conjugation.  
Suppose that $\lambda_1 \in \Lambda$ is not real, and that $\lambda_2 = \overline{\lambda_1} \in \Lambda$.  
There exists a polynomial $p(t) \in {\mathbf R}[t]$ of degree at most $r - 1$ 
such that 
\begin{equation}                    \label{perm:Lagrange-1}
p(\lambda_1) = i, \qquad p(\lambda_2) = -i 
\end{equation}
and
\begin{equation}                    \label{perm:Lagrange-2}
p(\lambda_j) = 0 \qquad \text{for all $j \in \{3,4,\ldots, r\}$.}
\end{equation}
\end{lemma}

\begin{proof}
By Lagrange interpolation, there exists a polynomial $p(t) \in  {\mathbf C}[t]$ of degree at most $r - 1$ 
that satisfies~\eqref{perm:Lagrange-1} and~\eqref{perm:Lagrange-2}.  
We have  
\[
p\left(\overline{\lambda_1}\right) = p(\lambda_2) = -i = \overline{i} = \overline{p(\lambda_1)}.
\]
and
\[
p\left(\overline{\lambda_2}\right) = p(\lambda_1) = i = \overline{-i} = \overline{p(\lambda_2)}.
\]
Also, $p\left(\overline{\lambda_j}\right) = 0 = \overline{p(\lambda_j)} $ 
for all $j \in \{3,4,\ldots, r\}$.
It follows from Lemma~\ref{perm:lemma:Lagrange-0} that $p(t) \in {\mathbf R}[t]$.
This completes the proof.  
\end{proof}

\begin{theorem}[Hermite-Sylvester]                           \label{perm:theorem:HermSyl}
The nonconstant polynomial $f(x)  \in {\mathbf R}[x]$ is real-rooted if and only if the quadratic form 
$Q_f$ constructed from the Hermite matrix $H_f$ is positive semidefinite.
\end{theorem} 

\begin{proof}
Let $f(x) \in {\mathbf R}[x]$ be a  polynomial of degree $n \geq 1$, and let $\Lambda$ 
be the set of distinct roots of $f(x)$.    
We have  $|\Lambda| = r \leq n$.  
Let $\Lambda = \{\lambda_1,\lambda_2,\ldots, \lambda_r\}$.
Because $f(x)$ has real coefficients, the set $\Lambda$ is closed under conjugation.  

Extend the sequence $\lambda_1,\lambda_2,\ldots, \lambda_r$ of $r$ distinct roots of $f(x)$ 
to the sequence $\lambda_1,\ldots,\lambda_r, \lambda_{r+1}, \ldots,  \lambda_n$ of $n$ roots of $f(x)$ 
with  multiplicity.  
Thus, for $j \in \{1,\ldots, r\}$, the root $\lambda_j$ with multiplicity $\mu_j$ appears $\mu_j$ times in this sequence.

Let $H_f = \bmat h_{i,j}\emat$ be the Hermite matrix constructed from the polynomial $f(x)$, 
and let $Q_f$ be the quadratic form constructed from $H_f$.  
For $\bsmallmat x_1 \\ \vdots \\ x_n \esmallmat \in {\mathbf R}^n$, we have 
\begin{align*}
Q_f(x_1,\ldots, x_n) & = \sum_{i=1}^n \sum_{j=1}^n h_{i,j} x_i x_j 
= \sum_{i=1}^n \sum_{j=1}^n m_{i+j-2} x_i x_j \\ 
& = \sum_{i=1}^n \sum_{j=1}^n \sum_{\ell=1}^n \lambda_{\ell}^{i+j-2} x_i x_j  
= \sum_{\ell=1}^n \left( \sum_{i=1}^n  x_i  \lambda_{\ell}^{i-1} \right) \left( \sum_{j=1}^n x_j \lambda_{\ell}^{j-1}  \right) \\
& = \sum_{\ell=1}^n\left( \sum_{j=1}^n x_j  \lambda_{\ell}^{j-1} \right)^2 = \sum_{\ell=1}^n p\left( \lambda_{\ell} \right)^2
\end{align*}
where 
\[
p(t) =  \sum_{j=1}^n x_j  t^{j-1} \in {\mathbf R}[t].  
\]
If $f(x)$ is real-rooted, then  $\lambda_{\ell} \in {\mathbf R}$ for all $\ell \in \{ 1,\ldots, n\}$, 
and so the sum $p\left( \lambda_{\ell} \right) = \sum_{j=1}^n x_j  \lambda_{\ell}^{j-1}$ is real 
and $ p\left( \lambda_{\ell} \right)^2 \geq 0$.  Thus, if $f(x)$ is real-rooted, then 
$Q_f(x_1,\ldots, x_n)  \geq 0$ and the quadratic form $Q_f$ is positive semidefinite.  

Suppose that $f(x)$ is not real-rooted.  
In this case, the set $\Lambda$ contains a complex number that is not real, 
 and $\Lambda$  also contains its complex conjugate.   
 We can assume that $\lambda_1 \in \Lambda$ is not real and that $\lambda_2 = \overline{\lambda_1}$. 
Note that $\mu_1 \geq 1$ is the multiplicity  of the root $\lambda_1$,   
and that $\mu_1$ is also the multiplicity  of the root $\lambda_2$.    
By Lemma~\ref{perm:lemma:Lagrange-1}, there exists a polynomial $p(t) \in {\mathbf R}[t]$ of degree $d-1 \leq r-1$  
such that 
\[
p(\lambda_1) = i, \qquad p(\lambda_2) = -i
\]
and 
\[
p(\lambda_{\ell}) = 0 \qquad\text{for $\ell = 3, \ldots, r$.}
\]
Because $r-1 \leq n-1$, we can write $p(t) = \sum_{j=1}^n x_j t^{j-1}$, 
where ${\mathbf x} = \bsmallmat x_1\\ \vdots \\ x_n \esmallmat \in {\mathbf R}^n$.   
Note that  $x_j =0$ for $j=d+1,\ldots, n$.  
We obtain  
\begin{align*}
Q_f(x_1,\ldots, x_n) 
& = \sum_{\ell=1}^n \left( \sum_{j=1}^n x_j  \lambda_{\ell}^{j-1} \right)^2  
= \sum_{\ell=1}^n p(\lambda_{\ell})^2 \\ 
& = \mu_1p(\lambda_1)^2 + \mu_1p(\lambda_2)^2 = \mu_1i^2 + \mu_1(-i)^2 \\
& = -2\mu_1 < 0.
\end{align*}
Thus, if $f(x)$ is not real-rooted, then the quadratic form $Q_f$ is not positive semidefinite.
This completes the proof.  
\end{proof}

\end{document}